\documentclass[12pt,leqno,psamsfonts,longtable,twoside]{amsart}

\usepackage{graphicx}
\usepackage{amsmath,amssymb,amsfonts,amsthm,amssymb,amscd,amstext,amsxtra,amsopn,array,url,verbatim,mathrsfs}
\usepackage{url}
\usepackage{booktabs}
\newcommand{\ra}[1]{\renewcommand{\arraystretch}{#1}}
\usepackage{longtable}
\usepackage{fullpage}
\usepackage{times}
\usepackage{dsfont}
\makeatletter
\@namedef{subjclassname@2010}{%
\textup{2010} Mathematics Subject Classification}
\makeatother
\newtheorem{theorem}{Theorem}[section]
\newtheorem{lemma}[theorem]{Lemma}

\newtheorem{corollary}[theorem]{Corollary}
\newtheorem{definition}[theorem]{Definition}
\numberwithin{equation}{section}

\theoremstyle{remark}

\renewcommand{\mod}[1]{{\ifmmode\text{\rm\ (mod~$#1$)}\else\discretionary{}{}{\hbox{ }}\rm(mod~$#1$)\fi}}

\renewcommand{\O}{{\mathcal{O}}}
\newcommand{\R}{{\mathbb{R}}}
\newcommand{\C}{{\mathbb{C}}}

\newcommand{\N}{{\mathbb{N}}}
\renewcommand{\Re}{{\mathfrak{Re}}}
\renewcommand{\Im}{{\mathfrak{Im}}}

\renewcommand{\a}{\alpha}
\renewcommand{\b}{\beta}
\newcommand{\g}{\gamma}

\renewcommand{\t}{\theta}
\newcommand{\s}{\sigma}
\renewcommand{\d}{\delta}

\begin{document} 

\title{New bounds for $\psi(x)$}
\author[L. Faber]{Laura Faber}
\address{Department of Mathematics and Computer Science, University of Lethbridge, 4401 University Drive, Lethbridge, Alberta, T1K 3M4 Canada}
\email{laura.faber2@uleth.ca}
\author[H. Kadiri]{Habiba Kadiri}
\address{Department of Mathematics and Computer Science, University of Lethbridge, 4401 University Drive, Lethbridge, Alberta, T1K 3M4 Canada}
\email{habiba.kadiri@uleth.ca}
\thanks{The first author was funded by a Chinook Research Award.
The second author was funded by ULRF Fund 13222.}
\thanks{Our calculations were done on the University of Lethbridge Number Theory Group Eudoxus machine, supported by an NSERC RTI grant.}
\begin{abstract}
In this article we provide new explicit Chebyshev's bounds for the prime counting function $\psi(x)$.
The proof relies on two new arguments: 
smoothing the prime counting function which allows to generalize the previous approaches,
and a new explicit zero density estimate for the zeros of the Riemann zeta function.
\end{abstract}
\date{\today}
\maketitle
\section{Introduction.}
\subsection{Main Theorem and History.}
We recall that $\psi(x)$ is the Chebyshev function given by
\[
 \psi(x)=\sum_{n\le x}\Lambda(n),
\ \text{with}\ 
 \Lambda(n)=\begin{cases}
		\log p & \text{ if  } n=p^k \text{ for } k\ge1, \\  
		0 & \text{ else.}
	    \end{cases}
\]
The Prime Number Theorem (PNT) is equivalent to
\[\psi(x) \sim x \ \text{ as }\ x\to\infty.\]
This estimate is a core tool in solving many problems in number theory and an explicit form of it turns out to be very useful in a wide range of problems.
In this article, we investigate explicit bounds (also known as Chebyshev's bounds) for the error term
\[E(x) = \left|\frac{\psi(x)-x}x \right|.\]
For instance, the main article of reference \cite{RR} in this subject is extensively used in various fields including Diophantine approximation, cryptography, and computer science. 
Moreover, breakthroughs concerning Goldbach's conjecture (see the work of Ramar\'e \cite{Ram0}, Tao \cite{Tao}, and Helfgott \cite{HH1} \cite{HH2}) rely on sharp explicit bounds for finite sums over primes. 
We combine a new explicit zero density estimate for $\zeta(s)$ and an optimized smoothing argument to prove
\begin{theorem} \label{thm:bigthm} 
Let $b_0\le 9 963$ be a fixed positive constant. Let $x\ge e^{b_0}$.
Then there exists $\epsilon_0>0$ such that 
$ E(x) \le \epsilon_0, $
where $\epsilon_0$ is given explicitly in \eqref{def-eps_0} and is computed in Table \ref{TablePlatt}.  
\end{theorem}
\begin{corollary}
 \label{bigcor} 
For all $x\ge e^{20}$, $ E(x) \le 5.3688\cdot10^{-4}$.
\end{corollary}
A classic explicit formula that relates prime numbers to non-trivial zeros of $\zeta$ is given by \cite[$\S 17$, $(1)$]{Dav}:
\begin{equation} \label{explicit} 
\psi(x)=x-\sum_{\rho}\frac{x^{\rho}}{\rho}-\log2\pi-\frac12\log(1-x^{-2}), 
\end{equation}
when $x$ is not a prime power. 
As the sum over the zeros is not absolutely convergent, it is impossible to directly use this formula to bound the error term $E(x)$.
To bypass this problem, the standard argument is to apply an explicit formula to an average of $\psi(x)$ on a small interval containing $[0,x]$. 

In $1941$ Rosser \cite[Theorem 12]{R} provides an explicit version of this proof.
In 1962 Rosser and Schoenfeld \cite[Theorem 28]{RS1} improve on this method by introducing further averaging.
Later results of Rosser and Schoenfeld \cite{RS2}, Dusart \cite{Dus0} \cite{Dus2}, and very recently Nazardonyavi and Yakubovich \cite{NY} all use the argument of \cite{RS1}.
They successively obtain smaller bounds for the error term as a consequence of improvements concerning the location of the non-trivial zeros of the Riemann zeta function, namely the verification of the Riemann Hypothesis up to a fixed height $H$, and an explicit zero-free region of the form $\Re s \ge 1-\dfrac1{R\log |\Im s|}\ \text{ and }\ |\Im s|\ge 2$, where $R$ is a computable constant.
On the other hand Theorem \ref{thm:bigthm} relies on new arguments.
We introduce a smooth weight $f$ and compare $\psi(x)$ to the sum
$ \mathscr{S}(x) = \sum_{n\ge1} \Lambda(n)f \Big(\frac{n}{x} \Big)$. 
In Section \ref{props-g} we choose $f$ in a close to optimal way so as to make the bound on $E(x)$ as small as possible.
We also observe that Rosser and Schoenfeld's averaging method is a special case of this smoothing method (see Section \ref{Rosser} for further discussion).
In Theorem \ref{Exp-Form} we establish a general explicit formula for $ \mathscr{S}(x)$.
A large contribution to the size of $E(x)$ arises from a sum over the non-trivial zeros of the form $\sum_{\rho}x^{\rho-1}F(\rho)$, where $F$ is the Mellin transform of $f$.
This sum is studied in Section \ref{BndError}. 
We split it so as to isolate zeros closer to the 1-line (say of real part larger than a fixed $\s_0$) as they contribute the most to the sum. 
In section \ref{using-density} we estimate this contribution by using for the first time explicit estimates for the zero density $N(\s_0,T)$ (as given in article \cite{Kad2}).  
This allows an extra saving over previous methods as they are of size between $\log T$ and $T$ smaller than $N(T)$.
Finally Theorem \ref{thm-bound-error} provides a general form for the bound of the error term $E(x)$.

We provide here a history of numerical improvements for Theorem \ref{thm:bigthm} in the case where $b_0=50$.
At the same time we mention which height $H$ and constant $R$ were used. 
{\small\begin{table}[hr]
\label{tab:T1}
\caption{For all $x\ge e^{50}$, $ E(x) \le \epsilon_0$.}
\begin{tabular}{| l | r l | r l | c |  }
\hline
Authors                                & $H$                     &                           & $R$                      &            & $\epsilon_0$ \\
\hline
Rosser \cite{R}                        & $1\,467$               & \cite{R}                  & $17.72$                  & \cite{R}   & $ 1.1900\cdot10^{-2}$\\
Rosser and Schoenfeld \cite{RS1}       & $21\,943$              & \cite{Leh1} \cite{Leh2}   & $17.5163\ldots$          & \cite{RS1} & $ 1.7202\cdot10^{-3}$\\
Rosser and Schoenfeld \cite{RS2}       & $1\,894\,438$          & \cite{RS2}                & $9.645908801$            & \cite{RS2} & $ 1.7583\cdot10^{-5}$\\
Dusart \cite{Dus0}                     & $545\,439\,823$        & \cite{VDL}                & $9.645908801$            &       \cite{RS2}     & $ 9.0500\cdot10^{-8}$\\
Dusart \cite{Dus2}*                    & $2\,445\,999\,556\,030$& \cite{Gou}*               & $5.69693$                & \cite{Kad} & $ 1.3010\cdot10^{-9}$\\
Nazardonyavi and Yakubovich \cite{NY}* & $2\,445\,999\,556\,030$&   \cite{Gou}*                         & $5.69693$                &    \cite{Kad}         & $ 1.3055\cdot10^{-9}$\\
Faber and Kadiri                       & $2\,445\,999\,556\,030$&   \cite{Gou}*                         & $5.69693$                &   \cite{Kad}          & $9.4602\cdot10^{-10}$\\
                                       & $30\,610\,046\,000$    & \cite{Pla} \cite{Pla0}    & $5.69693$                &        \cite{Kad}     & $2.3643\cdot10^{-9}$\\
\hline
\end{tabular}
\\
(* unpublished)
\end{table}}
\newline
Note that when we use the same values for $H$ and $R$ than \cite{Dus2} and \cite{NY}, our bounds for $E(x)$ are consistently smaller than theirs (for all $b_0$ except for $b_0=10\,000$ in the case of \cite{Dus2}).
\subsection{Zeros of the Riemann zeta function.}
We use the latest computations of Platt \cite{Pla0} \cite{Pla} concerning the verification of RH:
\begin{theorem}\label{RH}
Let $H = 3.061\cdot 10^{10}$. 
If $\zeta(s)=0$ at $0\le\Re(s)\le1$ and $0\le \Im(s) \le H$, then $\Re(s)=\frac12$.
\end{theorem}
Table \ref{TablePlatt} presents values of $\epsilon_0$ computed for this value of $H$.
Prior to the work of Platt, Gourdon \cite{Gou} announced a verification up to $H=2\,445\,999\,556\,030$.
We choose to use Platt's value of $H$ since his verification of RH is the most rigorous to date (he employs interval arithmetic).
Since other recent results (\cite{Dus2} and \cite{NY}) use Gourdon's $H$, we also give a version of Theorem \ref{thm:bigthm} based on his value (see Table \ref{TableGourdon}).

From \cite[Theorem 1.1]{Kad} we have the zero-free region:
\begin{theorem} \label{ZFR}
Let $R=5.69693$. 
Then there are no zeros of $\zeta(s)$ in the region
\[\Re s \ge 1-\frac1{R\log|\Im s|}\text{ and } |\Im s|\ge2.\]
\end{theorem}
Let $T\ge 2$ and $N(T)$ be the number of non-trivial zeros $\varrho =\beta+i\gamma$ in the region $0 \le \gamma \le T$ and $0\le\beta\le1$.
In $1941$, Rosser \cite[Theorem 19]{R} proved  
\begin{theorem}\label{Rosser41}
Let $T\ge 2$, 
\[
P(T)  = \frac T{2\pi}\log \frac T{2\pi} - \frac T{2\pi} + \frac{7}{8} ,\ 
R(T) = a_1 \log T + a_2 \log \log T + a_3,
\]  
and $\displaystyle{a_1 = 0.137}$, $\displaystyle{a_2 =0.443}$, $\displaystyle{a_3 = 1.588}$.
Then
\[ |N(T) - P(T) | \le  R(T).\]
\end{theorem}
We recall that $N(\s_0,T)$ is the number of non-trivial zeros in the region $\s_0 \le \Re s\le 1$ and $0\le \Im s\le T$.
In \cite{Kad2} the second author proved explicit upper bounds for $N(\s_0,T)$:
\begin{theorem}\label{main-density}
Let $3/5\le \s_0<1$. 
Then there exists constants $c_1,c_2,c_3$ such that, for all $T\ge H$,
\[ N(\s_0,T)  \le c_1  T + c_2   \log T + c_3.\]
\end{theorem}
The $c_i$'s depend on various (hidden) parameters and it is possible to choose these so as to make the above bound smaller when $T$ is asymptotically large or when it is close to $H$, the height of the numerical verification of RH. 
Table \ref{table2} at the end of this paper list values for the $c_i$'s in these respective cases.
For instance, it gives
\[
N(89/100,T)  \le 0.4617 T + 0.6644  \log  T -340\,272,
\]
which provides a saving of about $1/3 (\log T)$ compared to Theorem \ref{Rosser41}. 

When $T$ is near $H$, Theorem \ref{main-density} yields values for the $c_i$'s which provide a bound for $N(\s,T)$ of size about $\log H$.
For instance, it gives that $N(99/100,H) \le 78$ while Rosser's Theorem gives $5.2\cdot 10^{10}$.
\section{General form of an explicit inequality for $\psi(x)$.}
\subsection{Introducing a smooth weight $f$.} \label{smooth-f}
\begin{definition}\label{def-f}
Let $0<a<b, m\in \N$ and $m\ge 2$.
We define a function $f$ on $[a,b]$ by
$f(x)=1$ if $0\le x\le a$,
$f(x)=0$ if $x\ge b$, and
$f(x)=g\left(\frac{x-a}{b-a}\right)$ if $a\le x \le b$,
where $g$ is a function defined on $[0,1]$ satisfying
\begin{enumerate} 
\item[Condition 1:] $0\le g(x) \le 1\ $  for  $\ 0\le x \le 1$,
\item[Condition 2:] $g$ is an $m$-times differentiable function on $(0,1)$ such that for all $k=1, \ldots, m,$
\[g^{(k)}(0)=g^{(k)}(1)=0,\]
and there exist positive constants $a_k$ such that
\[ |g^{(k)}(x)|  \le a_k \ \text{ for all }\ 0< x<1.\]
\end{enumerate}
\end{definition}
We now consider
\begin{equation} \label{def-psitilde}
\mathscr{S}(x)= \sum_{n=1}^{\infty} \Lambda(n)f \Big(\frac{n}{x} \Big)
\ \text{ and }\ 
E_{\mathscr{S}}(x)= \left|\frac{\mathscr{S}(x)-x}{x}\right|.
\end{equation}
Let $\delta>0$. We denote $f^{-}$ and $f^{+}$ for the function $f$ defined above with the choices $a=1-\delta,b=1$ and $a=1,b=1+\delta$ respectively. 
We also define $\mathscr{S}^{-}$ and $\mathscr{S}^{+}$ the sums $\mathscr{S}$ associated to $f^-$ and $f^+$ respectively. 
Observe that
\begin{equation}\label{bnd-E}
\mathscr{S}^{-}(x) \le \psi(x) \le \mathscr{S}^{+}(x)
\ \text{ and }\ 
E(x) \le \max\left(E_{\mathscr{S}^{-}}(x),E_{\mathscr{S}^{+}}(x)\right).
\end{equation}
The Mellin Transform of $f$ is given by
\begin{equation} \label{eq:F}
F(s) = \int_{0}^{\infty} f(t) t^{s-1} dt. 
\end{equation}
We recall the property (see \cite[page 80, (3.1.3)]{KaPa}): 
if there exist $\a$ and $\b$ such that $\a<\b$ and, for every $\epsilon>0$, $f(x)=\O(x^{-\alpha-\epsilon})$ as $x\to 0$, and $f(x)=\O(x^{-\beta+\epsilon})$ as $x\to +\infty$, then $F$ is analytic in $\alpha<\Re s<\beta$.
It follows from our choice of $f$ that $F$ is analytic in $\Re s >0$.
Moreover, we have the inverse Mellin transform formula
\begin{align} \label{invMellin}
 f(t)=\frac1{2\pi i} \int_{2-i\infty}^{2+i\infty}F(s)t^{-s}ds.
\end{align}
Observe that
\[
\int_{a}^{b}|f^{(m+1)}(t)|t^{m+1}dt =  \frac1{(b-a)^m} \int_{0}^{1}|g^{(m+1)}(u)| \left((b-a)u+a\right)^{m+1}du.
\]
Let $k$ be a non-negative integer.
We define 
\begin{equation}\label{def-M}
M(a,b,k) = \int_{0}^{1}|g^{(k+1)}(u)| \left((b-a)u+a\right)^{k+1}du.
\end{equation}
We now record some properties of $F$. 
\begin{lemma}\label{propsF}
Let $0<a<b, m\in \N, m\ge2$.
Let $f$ and $g$ be functions as in Definition \ref{def-f}. 
\begin{enumerate}
  \item \label{part1} The Mellin transform $F$ of $f$ has a single pole at $s=0$ with residue 1 and is analytic everywhere else.  
  \item \label{part3} Let $s\in \C$ such that $\Re s\le 1$. Then $F$ satisfies
\begin{align} 
& \label{part2}
F(1)=a + (b-a)\int_{0}^{1} g(u)du ,\\ 
& \label{Fparts}
|F(s)| \le \frac{M(a,b,k)}{(b-a)^k|s|^{k+1}}, \ \text{ for all }\ k=0,\ldots, m.
\end{align}
\end{enumerate}
\end{lemma}
\begin{proof} 
The identity \eqref{part2} follows immediately from the definition of $f$. 
\\
We now use Condition 1 and Condition 2. 
We have $F(s) = \int_0^b f(t)t^{s-1}dt$ with $f'(x)=0$ for $0<x<a$. We integrate by parts once and observe that
$
F(s) = \frac{G(s)}{s},
$
where
\begin{equation}\label{def-G}
 G(s)=-\int_a^bf'(t)t^sdt
\end{equation}
is an entire function.
The residue of $F$ at $s=0$ is
$
G(0) =1.
$
\\
Let $\Re s \le 1$ and $k=0, \ldots,$ or $m$.
Inequality \eqref{Fparts} is obtained by integrating $F$ by parts $k+1$ times:
\begin{equation}
F(s) = \frac{(-1)^{k+1}}{s(s+1) \ldots (s+k)} \, \int_{a}^{b}  f^{(k+1)}(t) t^{s+k}  dt.
\end{equation}
We consider 
\[
G_m(s)=\int_a^b t^{s+m} f^{(m+1)}(t)dt.
\]
Since $f^{(i)}$ vanishes at both $a$ and $b$ for all $i=k,\ldots, m$,
we have
\begin{equation}
G_{m}(-k)  = (m-k)!(-1)^{m-k}\int_a^b f^{(k+1)}(t)dt 
 = (m-k)!(-1)^{m-k}(f^{(k)}(b)-f^{(k)}(a))=0. 
\end{equation}
Thus $F$ only has a pole at $s=0$ and is analytic everywhere else.
\end{proof} 
\subsection{An explicit formula for a smooth form of $\psi(x)$.} 
We use classical techniques to rewrite $\mathscr{S}(x)$ as a complex integral,
shift the integration contour to the left, and collect all the poles of 
the integrand so as to obtain a smooth analogue of the classical explicit formula \eqref{explicit}.
\begin{theorem}\label{Exp-Form}
Let $0<a<b, m\in \N, m\ge2$. 
Let $f$ be a function satisfying Definition \ref{def-f} and $F$ its Mellin transform. Then 
\[
 \mathscr{S}(x)  = xF(1) - \sum_{\rho} x^{\rho}F(\rho) - \frac{\zeta'}{\zeta}(0) - \sum_{n=1}^{\infty} x^{-2n} F(-2n),
\]
where $\rho$ runs through all the non-trivial zeros $\rho=\beta+i\gamma$ of the Riemann zeta function.
\end{theorem}
\begin{proof}
We insert \eqref{invMellin} in \eqref{def-psitilde}:
\[
\mathscr{S}(x) 
= \frac{1}{2 \pi i} \int_{2-i \infty}^{2+i \infty}x^sF(s) \Big(-\frac{\zeta'}{\zeta}(s)\Big) ds .
\]
Fix $k\in \R \backslash 2\N$ and $T\ge 2$ such that $T$ does not equal an ordinate of a zero of $\zeta$. 
Observe that the integrand has a pole at $s=0$ with residue $ -\frac{\zeta'}{\zeta}(0)$, a pole at $s=1$ with residue $xF(1)$, poles at the non-trivial zeros of zeta $\rho=\beta+i\gamma$ with residue $-x^{\rho}F(\rho)$, and poles at the trivial zeros of zeta $s=-2n, n\in \N,$ with residue $-x^{-2n}F(-2n)$.
We move the vertical line of integration extending from $2-iT$ to $2+iT$ to the line of integration extending from $-k-iT$ to $-k+iT$ so as to form the rectangle $\mathscr{R}$. 
Thus 
\[
\mathscr{S}(x) =  I_1(T,k)  + I_2(T,k) - I_3(T,k) -\frac{\zeta'}{\zeta}(0) + F(1) x  \\
 - \sum_{|\gamma|<T} x^{\rho} F(\rho) -\sum_{1\le n\le \frac{k}2} x^{-2n}F(-2n), 
\]
where $I_1,I_2,I_3$ are respectively integrating along the segments $[-k+iT,2+iT], [-k+iT,-k-iT], [-k-iT,2-iT]$.
It remains to prove that for each $j=1,2,3$, $\lim_{k,T\rightarrow+\infty}|I_j(T,k)|=0.$
We use the classical bounds (see \cite[page 108]{Dav})
\[
\Big| \frac{\zeta'}{\zeta}(\s+iT) \Big| \ll
\begin{cases}
  \log^2 T & \text{ if } -1 \le \s \le 2 , \\
 \log(|\s|+T) & \text{ if } -k\le  \s \le -1,
\end{cases}
\]
together with inequality \eqref{Fparts} for $F$, and obtain
\[ 
|I_1(T,k)| 
 \ll  \frac{ \log^2T}{T^{m+1}} \frac{x^2}{\log x} + \frac{\log T}{T^{m+1}} \frac1{x\log x} + \frac{x^{-T}}{T^{m-1}}  .
\]
We conclude that
$
 \lim_{k,T\rightarrow+\infty}|I_1(T,k)|=0.
$
Note that $I_3(T,k)=I_1(-T,k)$ converges to $0$ by a similar argument.
For $I_2(T,k)$, we combine \eqref{Fparts} with \cite[inequality (8)]{Dav}:
 \[
 |F(-k+it)| \Big| \frac{-\zeta'}{\zeta}(-k+it) \Big| \ll
\begin{cases}
 \frac{ \log k}{k^{m+1}}  & \text{ if } |t|\le\frac32, \\
 \frac{\log |t|}{|t|^{m+1}}  
& \text{ if } |t|>\frac32.
\end{cases}
\]
Thus
$
 |I_2(T,k)| 
 \ll x^{-k} \Big( \frac{\log k}{k^{m+1}}  +\frac{\log T}{T^{m}} \Big),
$
and
$
 \lim_{k,T\rightarrow +\infty}|I_2(T,k)|=0.
$
\end{proof}
\subsection{A general form of explicit bounds for $\psi(x)$.} \label{BndError}
We deduce from \eqref{Fparts} that
\[
\Big| \sum_{n=1}^{\infty} x^{-2n} F(-2n) \Big|
\le M(a,b,0) \sum_{n=1}^{\infty} \frac{x^{-2n}}{2n}
\le  \frac{M(a,b,0)}{2x^2}.
\]
Together with the above, \eqref{part2}, and $- \frac{\zeta'}{\zeta}(0)  =\frac{ \log (2\pi)}2$, 
it follows that
\begin{equation}\label{bd1}
E_{\mathscr{S}}(x)
 \le \Big|a-1 + (b-a)\int_0^1 g(u)du \Big| + \sum_{\rho} x^{\beta-1} |F(\rho)| +\frac{\log(2\pi)}2x^{-1} 
+  \frac{M(a,b,0)}{2}x^{-3} .
\end{equation}
To study the sum over the zeros, we introduce the notation
\begin{equation}\label{hypothese}
 \begin{split}
\ast\ &\ H>0 \text{ is such that if } \zeta(\b+i\g)=0 \text{ and } 0<\gamma<H, \text{ then } \b=1/2,\\
\ast\ &\  T_0>0 \text{ is such that } \sum_{0<\gamma<T_0}\gamma^{-1} \text{ can be directly computed},\\
\ast\ &\  T_1 \text{ is a parameter satisfying } T_0 < T_1 < H,\\  
\ast\ &\  R \text{ is a constant so that } \zeta(\s+it) \text{ does not vanish in the region }\\& \s \ge 1-\frac1{R\log |t|}\text{ and } |t|\ge2,\\
\ast\ &\  \s_0 \text{ is a parameter satisfying } 3/5\le \s_0 <1, \\
\ast\ &\  c_1>0,c_2>0,c_3<0 \text{ depend on }\s_0 \text{ so that }\\ & N(\s_0,T) \le c_1 T + c_2 \log T + c_3, \text{ for all } T\ge H.
 \end{split}
\end{equation}
Using the symmetry of the zeros of zeta and using the notation $\displaystyle{ \sideset{}{^*}\sum = \frac12\sum_{\beta=1/2}+\sum_{1/2<\beta<1}}$ we have:
\begin{equation}\label{sum1}
\sum_{\rho} x^{\beta-1}|F(\rho)|
=  \sideset{}{^*}\sum_{\gamma >0}\left( x^{\beta-1}+ x^{-\beta} \right)\left( |F(\rho)|+|F(\overline\rho)|\right).
\end{equation}
We now separate the zeros vertically at $H$: 
\begin{equation}\label{sum2}
 \sum_{\rho} x^{\beta-1}|F(\rho)|
= \Sigma_1 + \Sigma_2,
\end{equation}
with
\[
\Sigma_1= x^{-\frac12}\sum_{0< \gamma \le H}\left( |F(1/2+i\gamma)| + |F(1/2-i\gamma)| \right),\ 
\Sigma_2= \sideset{}{^*}\sum_{\gamma >H}\left( x^{\beta-1}+ x^{-\beta} \right) \left( |F(\rho)|+|F(\overline\rho)|\right). 
\]
We split $\Sigma_1$ vertically at $T_1$ and use \eqref{Fparts} to bound $|F(\rho)|$ with $k=0$ when $\gamma\le T_1$, and $k=m$ when $T_1 < \gamma\le H$  respectively.
Thus
\begin{equation} \label{s1tilde}
\Sigma_1 
\le 2 x^{-\frac12} \Big(  M(a,b,0)  \sum_{0< \gamma \le T_1} \frac1{\gamma}  + \frac{M(a,b,m)}{(b-a)^{m}} \sum_{T_1 < \gamma \le H} \frac1{\gamma^{m+1}} \Big) .
\end{equation} 
Moreover, we split the first sum at height $T_0\le T_1$ and denote $s_0$ a close upper bound for $\displaystyle\sum\limits_{\gamma \le T_0} \frac{1}{\gamma}$.
In \cite{RS2}, the authors use $T_0=158.84998$ and $s_0 = 0.8113925$.
We use here a computation of Darcy Best (personal communication) based on Odlyzko's list of zeros \cite{Odl}:
$T_0=1\,132\,491$ and $s_0 =  11.637732$.
\\
We use \eqref{Fparts} with $k=m$ for $\Sigma_2$ and split it horizontally at $\s_0$.
Together with the zero-free region given in Theorem \ref{ZFR} and the fact that $x^{\beta-1}+ x^{-\beta}$ increases with $\beta$,
we obtain 
\begin{equation}\label{sum3}
 \Sigma_2 
\\
\le 2\frac{M(a,b,m)}{(b-a)^{m}}  \Big( \left( x^{-(1-\sigma_0)}+ x^{-\sigma_0} \right) \sum_{\gamma > H} \frac1{\gamma^{m+1}}  
+  \sum_{\gamma > H, \s_0<\beta<1}\frac{ x^{-\frac1{R\log \gamma}}+ x^{-(1-\frac1{R\log H})} }{\gamma^{m+1}} 
\Big).
\end{equation}
We denote
\begin{equation}
\begin{split}\label{def-si}
&  s_1(T_1)= \sum_{0< \gamma\le T_1} \frac1{\gamma} ,\quad
   s_2(m,T_1) = \sum_{T_1 < \gamma\le H} \frac1{\gamma^{m+1}}, \quad
   s_3(m) = \sum_{\gamma> H} \frac1{\gamma^{m+1}}, \\
&  s_4(m,\s_0) = \sum_{\gamma > H, \s_0<\beta<1} \frac{1}{\gamma^{m+1}} ,\quad
   s_5(x,m,\s_0) = \sum_{\gamma > H, \s_0<\beta<1} \frac{ x^{-\frac1{R\log \gamma}} }{\gamma^{m+1}} .
\end{split} 
\end{equation}
We have
\begin{multline}\label{sum4}
\sum_{\rho} x^{\beta-1}|F(\rho)|
\le 2 \left(  M(a,b,0) s_1(T_1) + \frac{M(a,b,m)}{(b-a)^{m}} s_2(m,T_1) \right) x^{-\frac12} \\
 + 2  \frac{M(a,b,m)}{(b-a)^{m}} \left( \left( x^{-(1-\sigma_0)}+ x^{-\sigma_0} \right)s_3(m) +  x^{-(1-\frac1{R\log H})}s_4(m,\s_0) + s_5(x,m,\s_0) \right)  .
\end{multline}
We conclude by inserting \eqref{sum4} in \eqref{bd1}.
\begin{lemma}\label{bound-error-term}
Let $0<a<b, m\in \N$, with $m\ge2$.
Let $f$ be a function satisfying Definition \ref{def-f}.
Let $H, T_0,T_1,R$, and $\s_0$ satisfy \eqref{hypothese}. 
Then for all $x >0$,
$
  E_{\mathscr{S}}(x)  \le K(x,a,b,m,\s_0),
$
 where
\begin{multline}\label{def-K}
K(x,a,b,m,\s_0)
= \Big|a-1 + (b-a)\int_0^1 g(u)du \Big| \\
 + 2  \frac{M(a,b,m)}{(b-a)^{m}} \Big( \big( x^{-(1-\sigma_0)}+ x^{-\sigma_0} \big)s_3(m) +  x^{-(1-\frac1{R\log H})}s_4(m,\s_0) + s_5(x,m,\s_0) \Big) \\
+ 2 \Big(  M(a,b,0) s_0 +  M(a,b,0) s_1(T_1) + \frac{M(a,b,m)}{(b-a)^{m}} s_2(m,T_1) \Big) x^{-\frac12} \\
  +\frac{\log(2\pi)}2x^{-1} +  \frac{M(a,b,0)}{2}x^{-3} ,
\end{multline} 
and $M(a,b,m)$ and the $s_i$'s are defined in \eqref{def-M} and \eqref{def-si} respectively.
\end{lemma}
Note that for $a,b,m,\s_0$ fixed constants, $K(x,a,b,m,\s_0)$ decreases with $x$. Thus, for all $x\ge x_0$
\begin{equation}\label{bound-error-term-x0}
E_{\mathscr{S}}(x) \le K(x_0,a,b,m,\s_0).
\end{equation}

\subsubsection{Bounding $s_1(T_1)$, $s_2(m,T_1)$, and $s_3(m)$.}
We apply here a result from Rosser and Schoenfeld \cite{RS2}.
It uses explicit estimates for $N(T)$ as given in Theorem \ref{Rosser} to bound certain sums over the zeros of zeta.
\begin{lemma}\cite[Lemma 7]{RS2} \label{RS}
Let $1 < U \le V$, and let $\Phi(y)$ be nonnegative and differentiable for $U < y < V$.  Let $(W-y)\Phi'(y) \ge 0$ for   $U < y < V$, where $W$ need not lie in $[U,V]$.  Let $Y$ be one of $U,V,W$ which is neither greater than both the others or less than both the others.
Choose $j=0$ or $1$ so that  $(-1)^j(V-W) \ge 0$. Then 
\[
 \sum_{U < \gamma \le V} \Phi(\gamma) 
 \le \frac{1}{2 \pi} \int_{U}^{V} \Phi(y) \log \frac{y}{2 \pi} dy  
  + (-1)^j \Big(
 a_1 + \frac{a_2}{\log Y}
 \Big)
 \int_{U}^{V} \frac{\Phi(y)}{y} dy + E_j(U,V),
\]
where the error term $E_j(U,V)$ is given by 
\[  E_j(U,V)  
= (1+(-1)^j)R(Y) \Phi(Y) + (N(V)-P(V)-(-1)^jR(V))\Phi(V)\\ - (N(U)-P(U)+R(U)) \Phi(U).  
\]
\end{lemma}
\begin{corollary}\cite[Corollary of Lemma 7]{RS2} \label{corol}
If, in addition, $2 \pi < U$, then 
\[
 \sum_{U < \gamma \le V} \Phi(\gamma)
 \le (\frac{1}{2 \pi} + (-1)^j q(Y)) \int_{U}^{V} \Phi(y) \log \frac{y}{2 \pi} dy +E_j(U,V),\ \text{ where }\ 
q(y) = \frac{a_1 \log y + a_2}{y \log y \log(y/2 \pi)}. 
\]
\end{corollary}
Moreover, if $j=0$ and $W< U$, then
\begin{equation}
  \label{eq:Ejbd2}
  E_{0}(U,V) \le 2 R(U) \Phi(U). 
\end{equation}
%
We give details on how we apply Corollary \ref{corol} and \eqref{eq:Ejbd2} to $s_1, s_2$, and $s_3$. We take respectively 
\begin{itemize}
 \item 
$\Phi(y)=y^{-1}$, $U=T_0$, $V=T_1$, 
 \item 
$\Phi(y)=y^{-m-1}$, $U=T_1$, $V=H$, 
 \item 
$\Phi(y)=y^{-m-1}$, $U=H$, $V=\infty$.
\end{itemize}
In each case, $\Phi'(y)\le 0$ for all $y$, and we choose $W<U$, $Y=U$, and $j=0$. 
Since \begin{align*}
& \int_{T_0}^{T_1} \frac{\log \frac{y}{2 \pi}}{y} dy = \log (T_1/T_0)\, \log (\sqrt{T_1 T_0}/(2\pi))  ,\\
& \int_{U}^{V}  \frac{\log \frac{y}{2 \pi}}{y^{m+1}} dy = \frac{1+m \log(U/2 \pi)}{m^2 U^m} -\frac{1+m \log(V/2 \pi)}{m^2 V^m},
      \end{align*}
we obtain:
\begin{align}
& \label{def-B1}
s_1(T_1) \le B_1(T_1)= 
s_0+  \Big(\frac{1}{2 \pi} + q(T_0)\Big) \Big( \log (T_1/T_0)\, \log (\sqrt{T_1 T_0}/(2\pi)) \Big)  + \frac{2R(T_0)}{T_0}, \\
& \label{def-B2}
s_2(m,T_1)  \le  B_2(m,T_1)=
\Big(\frac{1}{2 \pi} + q(T_1)\Big) \Big( \frac{1+m \log(T_1/2 \pi)}{m^2 T_1^m}  -\frac{1+m \log(H/2 \pi)}{m^2 H^m} \Big) + \frac{2R(T_1)}{T_1^{m+1}}, \\
& \label{def-B3}
 s_3(m) \le B_3(m)=
\Big(\frac{1}{2 \pi} + q(H)\Big)  \frac{1+m \log(H/2 \pi)}{m^2 H^m}  +   \frac{2R(H)}{H^{m+1}}.
\end{align}
\subsubsection{Bounding $s_4(m,\s_0)$ and  $s_5(x,m,\s_0)$.}\label{using-density}
We assume here that $\Phi(y)=o(y)$ when $y \to \infty$, so as to ensure that $\lim_{y\to\infty} \Phi(y)N(\s_0,y)=0$. 
Since all non-trivial zeros of zeta have real part $1/2$ when $\gamma\le H$, then $N(\s_0,H)=0$ and we have the Stieltjes integral
\[
\sum_{\gamma\ge H , \beta >\sigma_0} \Phi(\gamma) = -\int_H^{\infty} N(\s_0,y) \Phi'(y) dy.
\]
\begin{lemma}\label{RSsigma}
Let $H,\s_0,c_1,c_2,c_3$ satisfy \eqref{hypothese}. 
Let $H < U \le V$, and let $\Phi(y)$ be non-negative and differentiable for $U < y < V$.
Assume $\Phi(y)=o(y)$ when $y \to \infty$ and 
$(W-y)\Phi'(y) \ge 0$ for all $U < y < V$, where $W$ need not lie in $[U,V]$. 
Let $Y$ be one of $U,V,W$ which is neither greater than both the others or less than both the others.
Then
\[
\sum_{U< \gamma <V , \beta >\sigma_0} \Phi(\gamma)
\le (c_1 Y + c_2 \log Y + c_3) \Phi(Y) - (c_1 V + c_2 \log V + c_3) \Phi(V) + \int_Y^{V} (c_1  + c_2 / y ) \Phi(y) dy.
\]
\end{lemma}
\begin{proof}
We have $0\le N(\s_0,y)\le c_1 y + c_2 \log y + c_3$. 
Our assumptions ensure us that $\Phi'(y)\ge 0$ if $U\le y\le Y$ and that $\Phi'(y)\le 0$ if $Y\le y\le V$.
Thus
\[
-\int_U^{V} N(\s_0,y) \Phi'(y) dy
\le -\int_Y^{V} (c_1 y + c_2 \log y + c_3) \Phi'(y) dy ,
\]
and we integrate by part to complete the proof.
\end{proof}
For $s_4(m,\s_0)$, we take 
$
\Phi(y) = \frac{1}{y^{m+1}} $,
$\Phi'(y) =- \frac{m+1}{y^{m+2}}$, 
$W<U=Y=H$, and $V=\infty$.
Thus 
\begin{equation}
\label{def-B4}
s_4(m,\s_0) \le B_4(m,H,\s_0)
= \left( c_1\big(1+\frac1m\big) 
+ c_2\frac{\log H}{H} 
+ \big(c_3 +\frac{c_2}{m+1}\big)\frac1H
 \right) \frac1{H^{m}}.
\end{equation}
For $s_5(x,m,\s_0)$, we apply Lemma \ref{RSsigma} with $U=H$, $ V=\infty$, 
$\Phi(y) = \phi_m(y) =\frac{x^{-\frac1{R\log y}}}{y^{m+1}} $,  
$\phi_m'(y) = \big( \frac{ \log x }{ R (\log y)^2 } - (m+1) \big) \frac{ \phi_m(y)}{y} $, and
\begin{equation}\label{def-W}
W=e^{\sqrt{\frac{\log x}{R(m+1)}}}.
\end{equation}
Let $J_m(Y)$ denote the integral
\[
J_m(Y) = \int_Y^{\infty} \phi_m(y) dy.
\]
We obtain
\begin{equation}\label{bd1-s5}
s_5(x,m,\s_0) 
\le (c_1 Y + c_2 \log Y + c_3)\phi_m(Y) + c_1J_{m}(Y) + c_2J_{m+1}(Y) ,  
\end{equation}
Let $ z>0,w\ge0$. We appeal to the theory of the following modified Bessel function 
\[
 K_{\nu}( z,w) = \frac12 \int_{w}^{\infty} t^{\nu-1} \exp\left(-\frac{ z}2 (t+1/t)\right) dt.
\]
We do the variable change $y=e^{\frac{z}{2m}t}$, take
$
  z = 2 \sqrt{ \frac{m \log x}{R} } $,
$  w =  \sqrt{ \frac{mR}{\log x} } \log Y = \frac{2m}{z} \log Y,
$
and recognize
\[
J_m(Y)
=\frac{z}{2m}  K_{1}(z,w).
\]
We use \cite[Lemma 4]{RS2} which asserts that if $w>1$ then 
\begin{equation}\label{bnd2-K1}
K_1(z,w) \le 
 Q_1(z,w) =\frac{ w^2}{z(w^2-1)}\exp\big(-z/2(w+1/w)\big) . 
 \end{equation}
We deduce for $J_m(Y)$ that
if $\log x < mR (\log Y)^2$, then
\begin{equation}\label{bnd1-Jm}
J_m(Y)
 \le \frac{ R }{2\log x} \frac{ (\log Y)^2 }{\big(\frac{m R }{\log x}\big)(\log Y)^2-1} Y^{-m} e^{- \frac{\log x}{R(\log Y)} }.
\end{equation}
In this case, we have $W<H$, $Y=H$. We insert \eqref{bnd1-Jm} in \eqref{bd1-s5} and obtain
\[
 s_5(x,m,\s_0) 
\le \big(c_1  + c_2 \frac{\log H}H + \frac{c_3}H\big) \frac{x^{-\frac1{R\log H}}}{H^{m}} + c_1J_{m}(H) + c_2J_{m+1}(H) ,  
\]
We conclude that if $\log x < mR (\log H)^2$ then
\begin{equation}
\label{def-B5}
s_5(x,m,\s_0) \le B_5(x,m,\s_0) 
= \Big(c_1  + c_2 \frac{\log H}H + \frac{c_3}H  + \big(c_1+\frac{c_2}H\big)\frac{ R }{2\log x} \frac{ (\log H)^2 }{\big(\frac{m R }{\log x}\big)(\log H)^2-1} \Big) \frac{x^{-\frac1{R\log H}}}{H^{m}}.
\end{equation}
\subsubsection{Main Theorem.}
We deduce a new bound for $K(x,a,b,m,\s_0)$ from \eqref{def-B1}, \eqref{def-B2}, \eqref{def-B3}, \eqref{def-B4}, and \eqref{def-B5}. 
Lemma \ref{bound-error-term} becomes
\begin{theorem}\label{thm-bound-error} 
Let $0<a<b, m\in \N$, with $m\ge2$. Let $f$ and $g$ be functions satisfying Definition \ref{def-f}, and $M(a,b,m)$ as defined in \eqref{def-M}. 
Let $H, T_0,T_1,R,\s_0,c_1,c_2,c_3$ satisfy \eqref{hypothese}.
Let $x_0$ be a positive constant satisfying $x_0< \exp(mR(\log H)^2)$.
Then for all $x\ge x_0$
\begin{multline}\label{bndError1}
E_{\mathscr{S}}(x)  
\le 
\big|a-1 + (b-a)\int_0^1 g(u)du \big|  
+  \frac{2 M(a,b,m)B_{5}(x_0,m,\s_0)}{(b-a)^{m}}  
+  \frac{2 M(a,b,m) B_3(m) }{(b-a)^{m}}  x_0^{-(1-\sigma_0)}
\\ +  \frac{2 M(a,b,m) B_3(m) }{(b-a)^{m}} x_0^{-\sigma_0}  
+   \frac{2 M(a,b,m)B_4(m,H,\s_0)  }{(b-a)^{m}}  x_0^{-(1-\frac1{R\log H})}  
\\+  \Big(  M(a,b,0) B_1(T_1) + \frac{M(a,b,m) B_2(m,T_1) }{(b-a)^{m}}\Big) x_0^{-\frac12}
  +\frac{\log(2\pi)}2 x_0^{-1} 
+  \frac{M(a,b,0)}{2} x_0^{-3} , 
\end{multline}
where the $B_i$'s are defined in \eqref{def-B1}, \eqref{def-B2}, \eqref{def-B3}, \eqref{def-B4}, and \eqref{def-B5}.
\end{theorem} 
\section{New explicit bounds for $\psi(x)$.} \label{Optim}
\subsection{Choosing the smooth function.} \label{props-g} 
We want to find a function $g$ satisfying Definition \ref{def-f} and so that the quotient $\frac{M(a,b,m)}{\int_0^1 g(u)du}$ is as small as possible. 
By the Cauchy-Schwarz inequality we have
\begin{equation} 
M(a,b,m)
\le \sqrt{\frac{b^{2m+3}-a^{2m+3}}{(b-a)(2m+3)}} \sqrt{\int_0^{1}\big(g^{(m+1)}(u)\big)^2du}. 
\label{eq:5.2}
\end{equation}
It follows from Calculus of Variations (see \cite[Chapter 2, \S11]{Fomin}) that the function $g$ optimizing the quotient
$
\frac{\sqrt{\int_0^{1}\big(g^{(m+1)}(u)\big)^2du}}{\int_0^1 g(u)du} 
$
is given by 
 \begin{equation} \label{def-g}
 g(x) = 1-\frac{(2m+1)!}{(m!)^2}\int_0^xt^m(1-t)^mdt.
 \end{equation}
We observe that our choice of kernel is a primitive of the one used in the context of short intervals containing primes by Ramar\'e \& Saouter \cite{RaSa}. 
This is not surprising as our object of study is $\sum_{n\ge1 }\Lambda(n)f(n/x)$, while theirs is essentially $\sum_{n\ge1} \Lambda(n)\left(f(n/y)-f(n/x)\right)$.
Since $y$ is close to $x$, this is approximately $\sum_{n\ge1} \Lambda(n) f'(n/x)$.
\\
With definition \eqref{def-g}, we find
  \begin{equation}\label{A3} 
\int_0^1 g(u)du
 = 1 - \frac{(2m+1)!}{(m!)^2} \int_0^1  t^{m}(1-t)^{m+1}  dt
= \frac{1}{2},
  \end{equation}
and
\begin{equation} \label{Mab_0} 
M(a,b,0)  = \frac{a+b}2.
\end{equation} 
We use \eqref{eq:5.2} to provide a simple bound for $M(a,b,m)$.
Since $g(1)=0,g(0)=1$, and $g^{(2m+2)}(x)=0$ for all $0<x<1$, integrating by parts $m$-times leads to
\[
\int_0^1(g^{(m+1)}(u))^2du 
 = (-1)^m \int_0^1 g^{(2m+1)}(u) \cdot g'(u)du \\
 = (-1)^{m+1}g^{(2m+1)}(0)
= \frac{(2m)!(2m+1)!}{(m!)^2}  .
\]
Thus \eqref{eq:5.2} becomes
\begin{equation}\label{Mabm}
 M(a,b,m) \le \lambda(a,b,m) = \sqrt{\frac{b^{2m+3}-a^{2m+3}}{(b-a)(2m+3)}}\cdot \frac{\sqrt{(2m)!(2m+1)!}}{m!}.
\end{equation}
From \eqref{def-g}, we recognize that
\[
g^{(m+1)}(u) =-\frac{(2m+1)!}{m!} P_m(1-2u),
\]
where $P_m$ is the $m^{th}$ Legendre polynomial as given by Rodrigues'formula (see \cite[formula (0.4)]{Ko}): 
\[
P_m(x) = \frac{1}{2^m m!} \frac{\partial^m}{\partial x^m} \left((x^2-1)^m\right).
\]
They can be written explicitly (see \cite[formula (0.2)]{Ko}):
\[
P_m(x) = \sum_{k=0}^m {m \choose k}^2 \left(\frac{x+1}2\right)^k\left(\frac{x-1}2\right)^{m-k}.
\]
These polynomials are well-known and are among the built-in functions of PARI/GP. 
Since the sign of $P_m$ alternates between its roots, $M(a,b,m)$ can be computed directly from 
\begin{equation}\label{identity-M}
M(a,b,m) = \frac{(2m+1)!}{m!}  \int_{0}^{1}|P_m(1-2u)| \left((b-a)u+a\right)^{m+1}du .
\end{equation}
\subsection{New explicit bounds for $\psi(x)$.} 
We rewrite Theorem \ref{thm-bound-error} with $g$ as chosen in \eqref{def-g}:
\begin{theorem}\label{thm-bound-error1b}\label{thm-bound-error2b}
Let $m\in \N, m\ge2$, $\d>0$, and the pair $(a,b)$ takes values $(1,1+\d)$ or $(1-\d,1)$.
Let $H, T_0,T_1,R,\s_0,c_1,c_2,c_3$ satisfy \eqref{hypothese}.
Let $b_0>0$ be a positive constant satisfying $b_0<  (m+1)R(\log H)^2$. Then for all $x\ge e^{b_0}$
 \begin{multline}\label{bndError2}
E_{\mathscr{S}}(x)  
 \le  \frac{\d}2  
+  \frac{2 M(a,b,m)B_{5}(e^{b_0},m,\s_0)}{\d^{m}}  
+  \frac{2 M(a,b,m) B_3(m) }{\d^{m}}  e^{-(1-\sigma_0)b_0}
\\ +  \frac{2 M(a,b,m) B_3(m) }{\d^{m}} e^{-\sigma_0 b_0}  
+   \frac{2 M(a,b,m)B_4(m,H,\s_0)  }{\d^{m}}  e^{-( 1-\frac1{R\log H})b_0}  
\\+  \Big( \frac{\d}2  B_1(T_1) + \frac{M(a,b,m) B_2(m,T_1) }{\d^{m}}\Big) e^{-b_0/2}
  +\frac{\log(2\pi)}2 e^{-b_0} 
+  \frac{M(a,b,0)}{2} e^{-3b_0} ,
 \end{multline}
 where $M(a,b,m)$ is given by \eqref{identity-M}, and the $B_i$'s are defined in \eqref{def-B1}, \eqref{def-B2}, \eqref{def-B3}, \eqref{def-B4}, and \eqref{def-B5}. \end{theorem} 
\subsection{Proof of Theorem \ref{thm:bigthm}.}
Let $b_0\ge 2$ be a fixed constant satisfying 
$b_0< 3R(\log H)^2$ (that is $b_0<9\,963 $ for $H=3.061\times10^{10}$ and $b_0<13\,906 $ for $H=2\,445\,999\,556\,030$). Let $x\ge e^{b_0}$.
We define 
\begin{multline}\label{def-eps}
 \epsilon(b_0,a,b,m,\s_0,T_1) = 
  \frac{\d}2  
+  \frac{2 M(a,b,m)B_{5}(e^{b_0},m,\s_0)}{\d^{m}} 
+  \frac{2 M(a,b,m) B_3(m) }{\d^{m}}  e^{-(1-\sigma_0)b_0}
\\ +  \frac{2 M(a,b,m) B_3(m) }{\d^{m}} e^{-\sigma_0 b_0}  
+   \frac{2 M(a,b,m)B_4(m,H,\s_0)  }{\d^{m}}  e^{-( 1-\frac1{R\log H})b_0}  
\\+  \Big( \frac{\d}2 B_1(T_1) + \frac{M(a,b,m) B_2(m,T_1) }{\d^{m}}\Big) e^{-b_0/2}
  +\frac{\log(2\pi)}2 e^{-b_0} 
+  \frac{M(a,b,0)}{2} e^{-3b_0}.
\end{multline}
The definition for $\epsilon_0$ follows directly from \eqref{bnd-E} and Theorem \ref{thm-bound-error1b}:
\begin{equation}\label{def-eps_0}
\epsilon_0 = \max\big( \epsilon(b_0,1,1+\d,m,\s_0,T_1),\epsilon(b_0,1-\d,1,m,\s_0,T_1)\big).
\end{equation}
To compute $ \epsilon(b_0,1,1+d,m,\s_0,T_1) $, we choose a value for $\s_0$ in Table \ref{table2}, an integer value larger than $2$ for $m$, and a value for $\d$ with up to $4$ significant digits.
Then we choose a value for $T_1$ which is either $T_0, H$ or so that it satisfies
\[
\frac{\d}2 B_1(T_1)  = \frac{M(1,1+\d,m) B_2(m,T_1) }{\d^{m}}.
\]
We do the same to compute $ \epsilon(b_0,1-\d,1,m,\s_0,T_1) $.
All values for $\s_0,m,$ and $\d$ are chosen to make $\epsilon_0$ as small as possible.
\subsection{Comparison with Rosser and Schoenfeld's method.} \label{Rosser}
\subsubsection{The smoothing argument.} 
The first step of their argument consists in studying $\psi(x)$ on average on a small interval around a large $x$ value. 
Let $x,\d>0$ with $x\notin \N$. Let $m\in \N$. 
It follows from the First Mean Value Theorem for Integrals applied to $h(z)=\psi(x+z)-(x+z)$ that
there exists $z\in$($0,\d x$) such that:
\[
h(z)+z  
\le \frac1{(\d/m x)^m}\int_0^{\d x/m}\ldots\int_0^{\d x/m} \left(h(y_1+\ldots+y_m)+(y_1+\ldots+y_m)\right)dy_1\ldots dy_m .
\]
(In order to make Rosser and Schoenfeld's article consistent with our setup, we replace their $\d$ with our $\d/m$.)
Implementing the explicit formula \eqref{explicit} in the right integrals together with the fact that $\psi(x+z) \le \psi(x)$ leads to \cite[Theorems 12 and 14]{R}:
\begin{equation} \label{bound}
E(x) \le \frac{ \delta }2 + 
\Sigma(m,\d,x)
+ \mathcal{O}(x^{-1}),
\end{equation}
with 
\[
\Sigma(m,\d,x)=\Big|\sum_{\rho} x^{\rho-1}  I_{m,\delta}(\rho) \Big| ,\ 
\text{ and }\ 
I_{m,\delta}(\rho)  = \frac{\sum_{j=0}^m(-1)^{j+m+1}\binom{m}{j}(1+j\delta/m)^{m+\rho}}{(\delta/m)^m  \rho (\rho+1)\ldots(\rho+m)}.
\]
We recall that we obtain \eqref{bound} with 
\[
\Sigma(m,\d,x)=\Big|\sum_{\rho} x^{\rho-1}  F(\rho) \Big| .
\]
We recognize that $ I_{m,\delta}$ is indeed the Mellin transform of
\[
\nu(t) = \frac1{m!}\sum_{j=0}^m (-1)^{j+m}\binom{m}{j}\Big(\frac{(1+j\delta/m)-t}{\delta/m}\Big)^m \mathds{1} \Big(\frac{t}{1+j\delta/m}\Big),
\]
where $\mathds{1} $ is the indicator function on $(0,1)$.
Instead we use the function $f$ given by Definition \ref{def-f} and \eqref{def-g}:
\[
f(x) = 1-\frac{(2m+1)!}{(m!)^2}\int_0^{\frac{x-1}{\d}}t^m(1-t)^mdt.
\]
We now compare the size of each Mellin transform.
Rosser establishes (see \cite[Theorem 15]{R}) that
\[
| I_{m,\delta}(\rho) |
 \le 
\frac{((1+\delta/m)^{m+1}+1)^m}{(\delta/m)^m|\gamma|^{m+1}} =  \frac{2^m m^m}{\delta^m  |\gamma|^{m+1}}(1+o(1))  ,
\]
while we have from \eqref{Fparts} and \eqref{Mabm}
\[|F(\rho)| \le \frac{M(1,1+\d,m)}{\d^m|\gamma|^{m+1}} \le\frac{\sqrt{(2m)!(2m+1)!}}{m!\d^m |\gamma|^{m+1}} (1+o(1)).
\]
It follows from Stirling Formula that the quotient 
$
\frac{|F(\rho)|}{| I_{m,\delta}(\rho) |}= \frac{\sqrt{(2m)!(2m+1)!}}{(2m)^m(m!)} 
$
decreases rapidly to $0$ as $m$ grows.
For instance it is $0.0083\ldots$ when we take $m=23$ for $b_0=50$.
\subsubsection{The new density of zeros.} 
When $x$ is large enough, the largest contribution to $\Sigma(m,\d,x)$ arises from 
\begin{equation}\label{mainSigma}
 \sum_{
\gamma>H , \s_0 < \b<1-\frac1{R\log \gamma}
} \frac{x^{-\frac1{R\log \gamma }} }{\gamma^{m+1}}. 
\end{equation}
Rosser and successive authors took $\s_0=1/2$ since only bounds for $N(T)$ were available.
Rosser and Schoenfeld find (see \cite[equations (3.4), (3.16) and (2.4)]{RS2}) that if $b_0\le2R\log^2H$ and $x\ge e^{b_0}$ then
\begin{equation}\label{eq-bnds5-rosser}
e^{\frac{b_0}{R\log H}} H^{m}  \sum_{ \gamma>H , 1/2<\b<1-\frac1{R\log \gamma}         
} \frac{x^{-\frac1{R\log \gamma }} }{\gamma^{m+1}}
\le  \frac{R (\log H)^3}{ 2\pi b_0\Big( \frac{mR(\log H)^2}{b_0}-1\Big)} (1+o(1)) .
\end{equation}
We are able to reduce significantly the contribution of the sum by using $\s_0$ closer to the limit of the zero-free region.
We establish that if $b_0\le 3 R\log^2H$ and $x\ge e^{b_0}$ then the above bound is replaced with
\[
\Big( c_1  + c_2 \frac{\log H}H + \frac{c_3}H \Big) + \Big(\big(c_1+\frac{c_2}H\big)\frac{ R }{2 b_0} \frac{ (\log H)^2 }{\big(\frac{m R }{b_0}\big)(\log H)^2-1} \Big) .
\]
When $\big(\frac{m R }{b_0}\big)(\log H)^2-1$ is large enough (for instance for $45 \le b_0\le 2000$ and $m\ge 10$), the main contribution arises from the above left expression.
We use the values for the $c_i$'s from the right column of Table \ref{table2} as they make $c_1H  + c_2  \log H + c_3$ small.
Otherwise, we use the values from the left column as they provide the smallest value for $c_1+\frac{c_2}H$.

\pagebreak
{\footnotesize
\begin{table}\centering
\caption{{\footnotesize For all $T\ge H$, $N(\s,T)\le c_1 T + c_2 \log T + c_3$.}}
\ra{1.3}
\begin{tabular}{@{}rcrrrcrrr@{}}
\toprule &
\phantom{abc} & 
\multicolumn{3}{c}{$c_1$ is small \label{table2}} & 
\phantom{abc}& 
\multicolumn{3}{c}{$c_1 H + c_2 \log H + c_3$ is small\label{table2b}}  
\\
\cmidrule{3-5} \cmidrule{7-9} 
$\s$   &
       &  
$c_1 $ & 
$c_2 $ & 
$c_3  \ $ & 
       & 
$c_1 $ & 
$c_2 $ & 
$c_3 \quad $\\
\midrule
\quad 0.60  &&  4.2288 & 2.2841  &  $-81\,673$ & & 28.6424 & 2.2841 & $-8.7674\cdot10^ {11}\quad $\\
0.65     && 2.4361 & 1.7965  &  $-97\,414$& &17.1679 & 1.3674 & $-5.2550 \cdot10^ {11}\quad $\\
0.70     && 1.4934 & 1.4609  &  $-136\,370$& &12.3778 & 0.9859 & $-3.7888 \cdot10^ {11}\quad $\\
0.75     & & 1.0031 & 1.1442  &  $-169\,449$& &9.6776 & 0.7708 & $-2.9622 \cdot10^ {11}\quad $\\
0.76     && 0.9355 & 1.0921  &  $-176\,604$& &9.2730 & 0.7386 & $-2.8384 \cdot10^ {11}\quad $\\
0.77     & & 0.8750 & 1.0437  &  $-184\,134$& &8.9009 & 0.7089 & $-2.7245 \cdot10^ {11}\quad $\\
0.78    & & 0.8205 & 0.9986  &  $-192\,120$& &8.5575 & 0.6816 & $-2.6194 \cdot10^ {11}\quad $\\
0.79    & & 0.7714 & 0.9566  &  $-200\,644$& &8.2396 & 0.6563 & $-2.5221 \cdot10^ {11}\quad $\\
0.80    & & 0.7269 & 0.9176  &  $-209\,795$& &7.9445 & 0.6328 & $-2.4317 \cdot10^ {11}\quad $\\
0.81  & & 0.6864 & 0.8812  &  $-219\,667$& &7.6698 & 0.6109 & $-2.3477 \cdot10^ {11}\quad $ \\
0.82    & & 0.6495 & 0.8473  &  $-230\,367$& &7.4135 & 0.5905 & $-2.2692 \cdot10^ {11}\quad $\\
0.83    & & 0.6156 & 0.8157  &  $-242\,009$& &7.1737 & 0.5714 & $-2.1958 \cdot10^ {11}\quad $\\
0.84   & & 0.5846 & 0.7862  &  $-254\,724$& &6.9490 & 0.5535 & $-2.1270 \cdot10^ {11}\quad $\\
0.85  & & 0.5561 & 0.7586  &  $-268\,658$& &6.7379 & 0.5367 & $-2.0624 \cdot10^ {11}\quad $\\
0.86   & & 0.5297 & 0.7327  &  $-283\,978$& &$6.5392$ & $0.5209$ & $-2.0016\cdot10^ {11}\quad $\\
0.87  & & 0.5053 & 0.7085  &  $-300\,872$& &6.3520 & 0.5059 & $-1.9443 \cdot10^ {11}\quad $ \\
0.88  &  & 0.4827 & 0.6857  &  $-319\,555$& &6.1751 & 0.4919 & $-1.8901 \cdot10^ {11}\quad $\\
0.89   & & 0.4617 & 0.6644   &  $-340\,272$& &6.0079 & 0.4785 & $-1.8389 \cdot10^ {11}\quad $\\
0.90   & & 0.4421 & 0.6443   &  $-363\,301$& &5.8494 & 0.4659 & $-1.7905 \cdot10^ {11}\quad $\\
0.91  &  & 0.4238 & 0.6253   &  $-388\,959$& &5.6991 & 0.4539 & $-1.7444 \cdot10^ {11}\quad $\\
0.92 && 0.4066 & 0.6075  &  $-417\,606$& & 5.5564 & 0.4426 & $-1.7007 \cdot10^ {11}\quad $\\
0.93 & & 0.3905 & 0.5906  &  $-449\,647$& &5.4206 & 0.4318 & $-1.6592 \cdot10^ {11}\quad $ \\
0.94  && 0.3754 & 0.5747   &  $-485\,543$& &5.2913 & 0.4215 & $-1.6196 \cdot10^ {11}\quad $\\
0.95 &  & 0.3612 & 0.5596   &  $-525\,807$& &5.1680 & 0.4116 & $-1.5819 \cdot10^ {11}\quad $\\
0.96  & & 0.3478 & 0.5452   &  $-571\,018$& &5.0503 & 0.4023 & $-1.5458 \cdot10^ {11}\quad $ \\
0.97  && 0.3352 & 0.5316  &  $-621\,815$& &4.9379 & 0.3933 & $-1.5114 \cdot10^ {11}\quad $\\
0.98 & & 0.3232 & 0.5187   &  $-678\,911$& &4.8304 & 0.3848 & $-1.4785 \cdot10^ {11}\quad $\\
0.99 & & 0.3118 & 0.5063   &  $-743\,087$& &4.7274 & 0.3766 & $-1.4470 \cdot10^ {11}\quad $ \\
\bottomrule
\end{tabular}
\end{table}
\ \\
To verify the values for the $c_i$'s, we refer the reader to \cite[Section 6]{Kad2}: we choose the parameters from this article to be $H=H_0-1$, $\s_0=0.522817$ for $\s=0.60$ and $\s_0=0.5208$ otherwise.
}
\ \newpage
{\footnotesize
\begin{longtable}{rrrrrrrrrrrrr}
\caption{Let $H=3.061\cdot 10^{10}$ and $b_0 \le 9\,963$.
For all $x\ge e^{b_0}$, $ E(x) \le \epsilon_0 $.
}
\label{TablePlatt}
\\
\hline
\phantom{} & 
$b_0$         & 
\phantom{a} &
$\s_0$        &
\phantom{a} & 
$m$           & 
\phantom{a} & 
$\delta$      & 
\phantom{a} & 
$T_1$         & 
\phantom{a} & 
$\epsilon_0$  &
\phantom{} \\
\hline
\endfirsthead
\multicolumn{13}{c}%
{\tablename\ \thetable{} -- continued from previous page} \\
\hline
\multicolumn{1}{c}{\phantom{}} &
\multicolumn{1}{c}{$b_0$  } &
\multicolumn{1}{c}{\phantom{a}} &
\multicolumn{1}{c}{$\s_0$ } &
\multicolumn{1}{c}{\phantom{a}} &
\multicolumn{1}{c}{$m$ } &
\multicolumn{1}{c}{\phantom{a}} &
\multicolumn{1}{c}{$\delta$ } &
\multicolumn{1}{c}{\phantom{a}} &
\multicolumn{1}{c}{$T_1$ } &
\multicolumn{1}{c}{\phantom{a}} &
\multicolumn{1}{c}{$\epsilon_0$ } &
\multicolumn{1}{c}{\phantom{}} 
\\ \hline 
\endhead
\hline \multicolumn{13}{r}{Continued on next page} \\ \hline
\endfoot
\hline 
\endlastfoot
&$20$     && $0.89$  && $4$   && $1.363 \cdot10^{-5}$ && $T_0$                  && $5.3688\cdot 10^{-4}$  &\\ &
$25$     && $0.89$  && $3$   && $7.256\cdot 10^{-6}$ && $T_0$                  && $4.8208\cdot 10^{-5}$ &\\&
$30$     && $0.89$ && $2$   && $2.811\cdot 10^{-6}$ && $T_0$                  && $5.6679\cdot 10^{-6}$ &\\ &
$35$     && $0.91$  && $3$   && $1.751\cdot 10^{-7}$ && $ 16\,739\,408 $       && $7.4457\cdot 10^{-7}$ &\\ &
$40$     && $0.92$   && $5$   && $2.142\cdot 10^{-8}$ && $ 245\,176\,468 $      && $8.6347\cdot 10^{-8}$ &\\ &
$45$     && $0.92$  && $13$  && $3.910\cdot 10^{-9}$ && $ 4\,085\,373\,679 $   && $1.0358\cdot 10^{-8}$ &\\ &
$50$    && $0.93$  && $23$ && $3.116\cdot10^{-9}$ && $ 9\,667\,437\,397 $  && $2.3643\cdot 10^{-9}$ &\\ &
$55$     && $0.93$  && $24$  && $3.105\cdot 10^{-9}$ && $ 10\,162\,544\,235 $  && $1.6783\cdot 10^{-9}$ &\\ &
$60$     && $0.93$  && $24$  && $3.099\cdot 10^{-9}$ && $ 10\,182\,181\,286 $  && $1.6191\cdot 10^{-9}$ &\\ &
$65$     && $0.94$   && $24$  && $3.093\cdot 10^{-9}$ && $ 10\,201\,894\,453 $  && $1.6114\cdot 10^{-9}$ &\\ &
$70$     && $0.94$   && $24$  && $3.087\cdot 10^{-9}$ && $ 10\,221\,684\,178 $  && $1.6081\cdot 10^{-9}$ &\\ &
$75$     && $0.94$   && $24$  && $3.082\cdot 10^{-9}$ && $ 10\,238\,234\,420 $  && $1.6052\cdot 10^{-9}$ &\\ &
$80$     && $0.95$   && $24$ && $3.225\cdot 10^{-9}$ && $ 10\,254\,838\,399 $   && $1.6025\cdot 10^{-9}$ &\\ &
$85$     && $0.95$   && $24$  && $3.071\cdot 10^{-9}$ && $ 10\,274\,834\,474 $   && $1.5997\cdot 10^{-9}$ &\\ &
$90$     && $0.95$  && $24$  && $3.066\cdot 10^{-9}$ && $ 10\,291\,557\,599 $   && $1.5969\cdot 10^{-9}$ &\\ &
$95$     && $0.95$   && $24$  && $3.061\cdot 10^{-9}$ && $ 10\,308\,335\,305 $   && $1.5942\cdot 10^{-9}$ &\\ &
$100$    && $0.95$  && $24$  && $3.056\cdot 10^{-9}$ && $ 10\,325\,167\,860 $   && $1.5916\cdot 10^{-9}$ &\\ &
$200$   && $0.97$  && $23$  && $2.960\cdot 10^{-9}$ && $ 10\,175\,863\,512 $   && $1.5422\cdot 10^{-9}$ &\\ &
$300$    && $0.97$  && $23$ && $2.866\cdot 10^{-9}$ && $ 10\,508\,919\,281 $   && $1.4953\cdot 10^{-9}$ &\\ &
$400$    && $0.98$  && $22$ && $2.769\cdot 10^{-9}$ && $ 10\,360\,124\,846 $   && $1.4476\cdot 10^{-9}$ &\\ &
$500$    && $0.98$   && $21$  && $2.674\cdot 10^{-9}$ && $ 10\,193\,677\,612 $   && $1.4006\cdot 10^{-9}$ &\\ &
$600$   && $0.98$   && $20$ && $2.579\cdot 10^{-9}$ && $ 10\,015\,840\,574 $   && $1.3543\cdot 10^{-9}$ &\\ &
$700$    && $0.98$  && $20$  && $2.492\cdot 10^{-9}$ && $ 10\,364\,671\,352 $   && $1.3081\cdot 10^{-9}$ &\\ &
$800$    && $0.98$  && $19$  && $2.397\cdot 10^{-9}$ && $ 10\,181\,118\,220 $   && $1.2616 \cdot 10^{-9}$ &\\ &
$900$    && $0.98$  && $18$  && $2.303\cdot 10^{-9}$ && $9\,979\,294\,107 $   && $ 1.2154\cdot 10^{-9}$ &\\ &
$1\,000$ && $0.98$  && $17$  && $2.209\cdot 10^{-9}$ && $ 9\,761\,696\,912 $   && $ 1.1695\cdot 10^{-9}$ &\\ &
$1\,500$ && $0.98$  && $14$  && $1.753\cdot 10^{-9}$ && $ 9\,882\,930\,682 $   && $ 9.3929\cdot 10^{-10}$ &\\ &
$2\,000$ && $0.99$  && $10$  && $1.293\cdot 10^{-9}$ && $ 9\,091\,299\,627 $   && $ 7.1125\cdot 10^{-10}$ &\\ &
$2\,500$ && $0.99$  && $6$   && $8.300\cdot10^{-10}$ && $ 7\,664\,220\,686 $   && $4.8137 \cdot 10^{-10}$ &\\ &
$3\,000$ && $0.99$  && $2$   && $3.000\cdot10^{-10}$ && $ 4\,992\,468\,020 $   && $2.2211 \cdot 10^{-10}$ &\\ &
$3\,500$ && $0.99$  && $2$   && $9.200\cdot10^{-11}$ && $ 14\,198\,916\,944 $  && $ 6.6209\cdot 10^{-11}$ &\\ &
$4\,000$ && $0.99$  && $2$   && $2.700\cdot10^{-11}$ && $ 26\,575\,655\,437 $  && $ 1.9689\cdot 10^{-11}$ &\\ &
$4\,500$ && $0.99$  && $2$   && $7.810\cdot10^{-12}$ && $ 30\,196\,651\,346 $  && $ 5.8563\cdot 10^{-12}$ &\\ &
$5\,000$ && $0.99$  && $2$   && $2.320\cdot10^{-12}$ && $ 30\,572\,809\,972 $  && $ 1.7434\cdot 10^{-12}$ &\\ &
$6\,000$ && $0.99$  && $2$   && $2.100\cdot10^{-13}$ && $ 30\,609\,694\,715 $  && $ 1.5457\cdot 10^{-13}$ &\\ &
$7\,000$ && $0.99$  && $2$   && $1.826\cdot10^{-14}$ && $ 30\,609\,997\,695 $  && $ 1.3693\cdot 10^{-14}$ &\\ &
$8\,000$ && $0.99$  && $2$   && $1.618\cdot10^{-15}$ && $ 30\,609\,999\,985 $  && $ 1.2135\cdot 10^{-15}$&\\ &
$9\,000$ && $0.99$  && $2$   && $1.434\cdot10^{-16}$ && $  H$                  && $1.0755\cdot 10^{-16}$ &\\ &
$9\,963$ && $0.99$  && $2$   && $1.390\cdot10^{-17}$ && $  H$                  && $9.5309\cdot 10^{-18}$ &\\ 
\hline
\end{longtable}
For $45\le b_0\le2000$ we use the values of $c_i$'s from the right column of Table \ref{table2}. We use the left values otherwise.
}
\newpage
{\footnotesize
\begin{longtable}{rrrrrrrrrrrrr}
\caption{
Let $H=2\,445\,999\,556\,030$ and $b_0 \le 13\,906$.
For all $x\ge e^{b_0}$, $E(x) \le \epsilon_0 $. }
\label{TableGourdon}
\\
\hline
\phantom{a} & 
$b_0$         & 
\phantom{a} &
$\s_0$        &
\phantom{a} & 
$m$           & 
\phantom{a} & 
$\delta$      & 
\phantom{a} & 
$T_1$         & 
\phantom{a} & 
$\epsilon_0$  &
\phantom{a} \\
\hline
\endfirsthead
\multicolumn{13}{c}%
{\tablename\ \thetable{} -- continued from previous page} \\
\hline
\multicolumn{1}{c}{\phantom{a}} &
\multicolumn{1}{c}{$b_0$} &
\multicolumn{1}{c}{\phantom{a}} &
\multicolumn{1}{c}{$\s_0$} &
\multicolumn{1}{c}{\phantom{a}} &
\multicolumn{1}{c}{$m$} &
\multicolumn{1}{c}{\phantom{a}} &
\multicolumn{1}{c}{$\delta$} &
\multicolumn{1}{c}{\phantom{a}} &
\multicolumn{1}{c}{$T_1$} &
\multicolumn{1}{c}{\phantom{a}} &
\multicolumn{1}{c}{$\epsilon_0$} &
\multicolumn{1}{c}{\phantom{a}} 
\\ \hline 
\endhead
\hline \multicolumn{13}{r}{Continued on next page} \\ \hline
\endfoot
\hline 
\endlastfoot
& $20$ && $0.88$ && $4$ && $1.363 \cdot 10^{-5}$ && $T_0$ && $5.3688\cdot 10^{-4}$ &\\&
$25$ && $0.89$ && $3$ && $7.256\cdot 10^{-6}$ && $T_0$ && $4.8208\cdot 10^{-5}$ &\\&
$30$ && $0.89$ && $2$ && $2.806\cdot 10^{-6}$ && $T_0$ && $5.6646\cdot 10^{-6}$ &\\&
$35$ && $0.90$ && $2$ && $1.604\cdot 10^{-7}$ && $11\,360\,452$ && $7.0190\cdot 10^{-7}$ &\\&
$40$ && $0.91$ && $3$ && $1.600\cdot 10^{-8}$ && $174\,242\,715$ && $8.0214\cdot 10^{-8}$ &\\&
$45$ && $0.92$ && $4$ && $1.613\cdot 10^{-9}$ && $2\,393\,630\,483$ && $8.6997\cdot 10^{-9}$ &\\&
$50$ && $0.93$ && $7$ && $2.058\cdot 10^{-10}$ && $36\,960\,925\,828$ && $9.4602\cdot 10^{-10}$ &\\&
$55$ && $0.96$ && $21$ && $5.079\cdot 10^{-11}$ && $532\,313\,030\,046$ && $1.1243\cdot 10^{-10}$ &\\&
$60$ && $0.96$ && $28$ && $4.807\cdot 10^{-11}$ && $770\,935\,427\,426$ && $3.2156\cdot 10^{-11}$ &\\&
$65$ && $0.96$ && $29$ && $4.801\cdot 10^{-11}$ && $801\,857\,986\,418$ && $2.5430\cdot 10^{-11}$ &\\&
$70$ && $0.96$ && $29$ && $4.795\cdot 10^{-11}$ && $802\,859\,999\,396$ && $2.4849\cdot 10^{-11}$ &\\&
$75$ && $0.96$ && $29$ && $4.789\cdot 10^{-11}$ && $803\,864\,521\,532$ && $2.4773\cdot 10^{-11}$ &\\&
$80$ && $0.97$ && $29$ && $4.783\cdot 10^{-11}$ && $804\,871\,562\,262$ && $2.4738\cdot 10^{-11}$ &\\&
$85$ && $0.97$ && $29$ && $4.777\cdot 10^{-11}$ && $805\,881\,131\,075$ && $2.4707\cdot 10^{-11}$ &\\&
$90$ && $0.97$ && $29$ && $4.771\cdot 10^{-11}$ && $806\,893\,237\,503$ && $2.4677\cdot 10^{-11}$ &\\&
$95$ && $0.97$ && $29$ && $4.765\cdot 10^{-11}$ && $807\,907\,891\,129$ && $2.4647\cdot 10^{-11}$ &\\&
$100$ && $0.97$ && $29$ && $4.759\cdot 10^{-11}$ && $808\,925\,101\,582$ && $2.4618\cdot 10^{-11}$ &\\&
$200$ && $0.98$ && $28$ && $4.647\cdot 10^{-11}$ && $797\,441\,603\,800$ && $2.4065\cdot 10^{-11}$ &\\&
$300$ && $0.98$ && $28$ && $4.546\cdot 10^{-11}$ && $815\,133\,603\,120$ && $2.3543\cdot 10^{-11}$ &\\&
$400$ && $0.98$ && $27$ && $4.440\cdot 10^{-11}$ && $802\,199\,639\,823$ && $2.3021\cdot 10^{-11}$ &\\&
$500$ && $0.98$ && $26$ && $4.334\cdot 10^{-11}$ && $788\,664\,950\,273$ && $2.2506\cdot 10^{-11}$ &\\&
$600$ && $0.98$ && $26$ && $4.237\cdot 10^{-11}$ && $806\,692\,808\,636$ && $2.1998\cdot 10^{-11}$ &\\&
$700$ && $0.99$ && $25$ && $4.131\cdot 10^{-11}$ && $792\,643\,976\,191$ && $2.1480\cdot 10^{-11}$ &\\&
$800$ && $0.99$ && $25$ && $4.032\cdot 10^{-11}$ && $ 812\,075 \,384 \,439$ && $2.0969 \cdot 10^{-11}$ &\\&
$900$ && $0.99$ && $23$ && $3.918\cdot 10^{-11}$ && $ 762\,588 \,970 \,852$ && $ 2.0443\cdot 10^{-11}$ &\\&
$1\,000$ && $0.99$ && $23$ && $3.818\cdot 10^{-11}$ && $ 782\,528 \,018 \,219$ && $ 1.9921\cdot 10^{-11}$ &\\&
$1\,500$ && $0.99$ && $20$ && $3.303\cdot 10^{-11}$ && $ 774\, 756\,126 \,279$ && $ 1.7342\cdot 10^{-11}$ &\\&
$2\,000$ && $0.99$ && $17$ && $2.788\cdot 10^{-11}$ && $764 \,936 \,897 \,224$ && $ 1.4762\cdot 10^{-11}$ &\\&
$2\,500$ && $0.99$ && $14$ && $2.272\cdot 10^{-11}$ && $ 752\,424 \,086 \,843$ && $1.2118 \cdot 10^{-11}$ &\\&
$3\,000$ && $0.99$ && $11$ && $1.755\cdot 10^{-11}$ && $ 735\,757 \,894 \,330$ && $9.5728 \cdot 10^{-12}$ &\\&
$3\,500$ && $0.99$ && $7$ && $1.209\cdot 10^{-11}$ && $ 618\, 567\,513 \,247$ && $ 6.9073\cdot 10^{-12}$ &\\&
$4\,000$ && $0.99$ && $4$ && $6.800\cdot 10^{-12}$ && $ 533\,755 \,825 \,076$ && $ 4.2115\cdot 10^{-12}$ &\\&
$4\,500$ && $0.99$ && $2$ && $2.300\cdot 10^{-12}$ && $ 576\,348 \,240 \,050$ && $ 1.6858\cdot 10^{-12}$ &\\&
$5\,000$ && $0.99$ && $2$ && $8.400\cdot 10^{-13}$ && $ 1\,334\,194\,702\,027$ && $6.0522\cdot 10^{-13}$ &\\&
$6\,000$ &&$0.99$ && $2$ && $1.036\cdot 10^{-13}$ && $ 2\,401\,904 \,005 \,983$ && $ 7.7686\cdot 10^{-14}$ &\\&
$7\,000$ &&$0.99$ && $2$ && $1.332\cdot 10^{-14}$ && $ 2\,445\,250 \,025 \,818$ && $ 9.9890\cdot 10^{-15}$ &\\&
$8\,000$ &&$0.99$ && $2$ && $1.713\cdot 10^{-15}$ && $2\,445 \,987 \,153 \,821$ && $ 1.2845\cdot 10^{-15}$&\\& 
$9\,000$ &&$0.99$ && $2$ && $2.202\cdot 10^{-16}$ && $2\,445 \,999 \,351 \,095$ && $1.6516 \cdot 10^{-16}$ &\\&
$10\,000$ &&$0.99$ && $2$ && $2.830\cdot 10^{-17}$ && $ 2\,445\,999 \,552 \,648 $ && $ 2.1236\cdot 10^{-17}$ &\\&
$13\,900$   &&$0.99$ && $2$ && $9.502 \cdot 10^{-21}$ && $H$ && $7.1265 \cdot 10^{-21}$ &\\
\hline
\end{longtable}
We only use the values of $c_i$'s from the left column of Table \ref{table2}.
}\newpage
\end{document}